\documentclass[11pt]{article}
\usepackage{amsfonts,amsbsy,amssymb,amsmath,graphicx}
\usepackage{geometry}
\usepackage[latin1]{inputenc}
\usepackage[T1]{fontenc} 
\usepackage{fancybox} % for shadow and Bitemize
\usepackage{alltt}
\usepackage{graphicx}
\usepackage[colorlinks,hyperindex,bookmarks,linkcolor=red,citecolor=blue,urlcolor=blue]{hyperref}
\usepackage{fancyhdr}
\usepackage{eso-pic, xcolor,graphicx}
\usepackage{tikz}
\usepackage{bbm}
\usepackage{natbib}

\newtheorem{theorem}{Theorem}[section]

\newtheorem{corollary}[theorem]{Corollary}

\newtheorem{definition}[theorem]{Definition}

\newtheorem{lemma}[theorem]{Lemma}

\newtheorem{proposition}[theorem]{Proposition}
\newtheorem{remark}[theorem]{Remark}

\usepackage{wrapfig}
\usepackage{epsf}

\newenvironment{proof} {\noindent { \textbf{Proof.}} } { \hfill \fbox{~} \\ } 
\def\1{1\kern-.20em {\rm l}}

\newcommand{\ER}{\ensuremath{\mathbb{R}}}

\def\S{{\mathcal{S}}}
\def\F{{\mathcal{F}}}
\def\G{{\mathcal{G}}}

\def\i{{\mathcal{I}}} 
 
\def\E{{\mathbb{E}}}

\def\P{{\mathbb{P}}}

\def\a{{\alpha}}
\begin{document}
\title {Kernel-smoothed conditional quantiles of randomly censored functional stationary ergodic data}

% Asymptotic analysis of a kernel conditional quantiles estimator for functional stationary ergodic data under random censorship
%\title {\sc Nonparametric conditional quantiles estimation for functional stationary ergodic data under random censorship}

\author{{\sc Mohamed Chaouch}$^{1,}\footnote{corresponding author}$ \quad and \quad {\sc Salah Khardani}$^2$ \\
$^1$ Centre for the Mathematics of Human Behaviour\\Department of Mathematics and
Statistics\\University of Reading- UK\\
$^2$Universit\'e Lille Nord de France- France\\ 
email : $m.chaouch$@reading.ac.uk, $khardani$@lmpa.univ-littoral.fr}

\maketitle

\begin{abstract}

 This paper, investigates the conditional quantile estimation of a scalar random response and a functional random covariate (i.e. valued in some infinite-dimensional space) whenever {\it functional stationary ergodic data with random censorship} are considered. We introduce a kernel type estimator of the conditional quantile function. We establish the strong consistency with rate of this estimator as well as the asymptotic normality which induces a confidence interval that is usable in practice since it does not depend on any unknown quantity. An application to electricity peak demand interval prediction with censored smart meter data is carried out to show the performance of the proposed estimator.

\bigskip

\end{abstract}

\section{Introduction }

\noindent Functional data analysis is a branch of statistics that has been the object of
many studies and developments these last years. This kind of data appears
in many practical situations, as soon as one is interested in a continuous-time phenomenon
for instance. For this reason, the possible application fields propitious
for the use of functional data are very wide: climatology, economics, linguistics,
medicine, \dots. Since the  works of \cite{ramsay1991}, many developments have been investigated,
in order to build theory and methods around functional data, for instance
how it is possible to define the regression function and the quantile regression function of functional data,
what kind of model it is possible to consider with functional data.
\noindent The study of statistical models for infinite dimensional (functional) data has been the subject of several  works in the recent statistical literature. We refer to \cite{bosq2000}, \cite{ramsay2005} in the parametric model and the monograph by \cite{ferraty2006} for the nonparametric case. There are many results for nonparametric
models. For instance, \cite{ferraty2004} established the strong consistency
of kernel estimators of the regression function when the explanatory variable is functional
and the response is scalar, and their study is extended to non standard regression problems
such as time series prediction or curves discrimination by \cite{ferraty2002} and \cite{ferraty2003}. The
asymptotic normality result for the same estimator in the alpha-mixing case has been obtained
by \cite{masry}.

\noindent In addition to the regression function, other statistics such as quantile and mode regression could be with interest for both sides theory and practice. Quantile regression is a common way to describe the dependence structure between a response variable $Y$ and some covariate $X$. Unlike the regression function (which is defined as the conditional mean) that relies only on the central tendency of the data, conditional quantile function allows the analyst to estimate the functional dependence
between variables for all portions of the conditional distribution of the response
variable. Moreover, quantiles are well-known by their robustness to heavy-tailed error distributions and outliers which allows to consider them as a useful alternative to the regression function.

Conditional quantiles for scalar response and a scalar/multivariate covariate have received considerable interest in the statistical literature. For completely observed data, several nonparametric approaches have been proposed, for instance, \cite{gannoun2003} introduced a smoothed estimator based on double kernel and local constant kernel methods and \cite{berlinet} established its asymptotic normality.  Under random censoring, \cite{gannoun2005} introduced a local linear (LL) estimator of quantile regression (see \cite{koenker} for the definition) and \cite{elghouch} studied the same LL estimator.  \cite{OS} constructed a kernel estimator of the conditional quantile under independent and identically distributed (i.i.d.) censorship model and established its strong uniform convergence rate. \cite{liang} established the strong uniform convergence (with rate) of the conditional quantile function under $\alpha$-mixing assumption. 

Recently, many authors are interested in the estimation of conditional quantiles for a scalar response and functional covariate. \cite{ferraty2005} introduced a nonparametric estimator of conditional quantile defined as the inverse of the conditional cumulative distribution function when the sample is considered as an $\alpha$-mixing sequence. They stated its rate of almost complete consistency and used it to forecast the well-known El Ni\~no time series and to build confidence prediction bands. \cite{ezzahrioui} established the asymptotic normality of the kernel conditional quantile estimator under $\alpha$-mixing assumption. Recently, and within the same framework, \cite{daboniang} provided the consistency in $L^p$ norm of the conditional quantile estimator for functional dependent data.

In this paper we investigate the asymptotic properties of the conditional quantile function of a scalar response and functional covariate when data are randomly censored and assumed to be sampled from a stationary and ergodic process. Here, we consider a model in which the response variable is censored but not the covariate. Besides the infinite dimensional character of the data, we avoid here the widely used strong mixing condition and its variants to measure the dependency and the very involved probabilistic calculations that it implies. Moreover, the mixing properties of a number of well-known processes are still open questions. Indeed, several models are given in the literature where mixing properties are still to be verified or even fail to hold for the process they induce. Therefore, we consider in our setting the ergodic property to allow the maximum possible generality with regard to the dependence setting. Further motivations to consider ergodic data are discussed in \cite{laib2010} where details defining the ergodic property of processes are also given. As far as we know, the estimation of conditional quantile combining censored data, ergodic theory and functional data has not been studied in statistical literature. 

The rest of the paper is organized as follows. Section \ref{S2} introduces the kernel estimator of the conditional quantile under ergodic and random censorship assumptions. Section \ref{S3} formulates main results of strong consistency (with rate) and asymptotic normality of the estimator. An application to peak electricity demand interval prediction with censored smart meter data is given in Section \ref{S4}. Section \ref{S5} gives proofs of the main results. Some preliminary lemmas, which are used in the proofs of the main results, are collected in Appendix.

%\noindent Conditional quantile function is a useful tool in statistical analysis. It has been used
%in exploratory data analysis, applied statistics, reliability and survival analysis (See,
%for example, Reid (1981), Slud et al. (1984), Su and Wei (1993), Nair et al. (2008),
%Nair and Sankaran (2009). The concept of quantiles has been used by
%Peng and Fine (2007) for modelling competing risk models.\\
% A frequent problem in survival data analysis is censoring, which may be
%due to different causes: in econometrics censoring can be due to the loss of
%some subjects under study; in clinical trials censoring can be caused by the end
%of the follow-up period; in ecology or environmental studies, single or multiple
%detection limits lead to censored observations. \\

\section{Notations and definitions}\label{S2}
\subsection{Conditional quantiles under random censorship}
In the censoring case, instead of observing the lifetimes $T$ 
(which has a continuous distribution function (df) ) 
we observe the censored lifetimes of
items under study. That is, assuming that $(C_i)_{1\leq i \leq n}$ is a sequence of i.i.d.
censoring random variable (r.v.) with common unknown continuous df $G$.

%Let $(\Omega, \mathcal{A},\P)$ be a probability space and let $T$ be a real-valued response variable defined 
%on $(\Omega, \mathcal{A},\P)$ with continuous distribution function (df) $F$. Let $(T_i)_{1\leq i \leq n}$ be a sample of random variables (rvs) distributed as $T$. Assume that $(C_i)_{1\leq i \leq n}$ is a sample of independent and identically distributed (iid) censoring rvs with common df $G$. 
Then in the right censorship model, we only observe the $n$ pairs $(Y_i, \delta_i)$ with
\begin{eqnarray}
Y_i = T_i \wedge C_i \quad \mbox{and} \quad \delta_i = \1_{\{T_i \leq C_i \}}, \quad 1\leq i \leq n,
\end{eqnarray}
where $\1_A$ denotes the indicator function of the set $A$. 

To follow the convention in biomedical studies and as indicated before, we
assume that $(C_i)_{1\leq i \leq n}$ and  $(T_i, X_i)_{1\leq i \leq n}$ are independent; this condition
is plausible whenever the censoring is independent of the modality of
the patients.

%We suppose that $(T_i)_{1\leq i \leq n}$ and $(C_i)_{1\leq i \leq n}$ are independent which ensures the identifiability of the model. Moreover, in this case the distribution $L$ of $Y_1$ satisfies $1-L = (1-F)(1-G).$

Let $(X, T)$ be $E\times\ER$-valued random elements, where $E$ is some semi-metric abstract space. Denote by $d(\cdot, \cdot)$ a semi-metric associated to the space $E$. Suppose now that we observe a sequence $(X_i, T_i)_{i\geq 1}$ of copies of $(X,T)$ that we assume to be {\it stationary} and {\it ergodic}.  For $x\in E$, we denote the conditional probability distribution of $T$ given $X=x$ by:
\begin{eqnarray}\label{defF}
\forall t\in\ER, \quad F(t\mid x) = \P\left( T\leq t \mid X=x\right).
\end{eqnarray}
We denote the conditional quantile, of order $\alpha\in (0,1)$, of $T$ given $X=x$, by
\begin{eqnarray}\label{defQ}
q_\alpha(x) = \inf\{ t : F(t\mid x) \geq \alpha \}.
\end{eqnarray}

%Let $(\Omega, \mathcal{A},\P)$ be a probability space and let $T$ be a real-valued response variable defined 
%on $(\Omega, \mathcal{A},\P)$ with continuous distribution function (df) $F$. Let $(T_i)_{1\leq i \leq n}$ be a sample of random variables (rvs) distributed as $T$. Assume that $(C_i)_{1\leq i \leq n}$ is a sample of independent and identically distributed (iid) censoring rvs with common df $G$. Then in the right censorship model, we only observe the $n$ pairs $(Y_i, \delta_i)$ with
%\begin{eqnarray}
%Y_i = T_i \wedge C_i \quad \mbox{and} \quad \delta_i = \1_{\{T_i \leq C_i \}}, \quad 1\leq i \leq n,
%\end{eqnarray}
%where $\1_A$ denotes the indicator function of the set $A$. We suppose that $(T_i)_{1\leq i \leq n}$ and $(C_i)_{1\leq i \leq n}$ are independent which ensures the identifiability of the model. Moreover, in this case the distribution $L$ of $Y_1$ satisfies $1-L = (1-F)(1-G).$
%
%Let $(X, T)$ be $E\times\ER$-valued random elements, where $E$ is some semi-metric abstract space. Denote by $d(\cdot, \cdot)$ a semi-metric associated to the space $E$. Suppose now that we observe a sequence $(X_i, T_i)_{i\geq 1}$ of copies of $(X,T)$ that we assume to be {\it stationary} and {\it ergodic}.  For $x\in E$, we denote the conditional probability distribution of $T$ given $X=x$ by:
%\begin{eqnarray}\label{defF}
%\forall t\in\ER \quad F(t\mid x) = \P\left( T\leq t \mid X=x\right).
%\end{eqnarray}
%We denote the conditional quantile of order $\alpha\in (0,1)$ by:
%\begin{eqnarray}\label{defQ}
%q_\alpha(x) = \inf\{ y : F(y\mid x) \geq \alpha \}.
%\end{eqnarray}

We suppose that, for any fixed $x\in E$, $F(\cdot\mid x)$ be continuously differentiable real function, and admits a unique conditional quantile.

Let $\a \in (0,1)$, we will consider the problem of estimating the parameter $q_\a(x)$ which satisfies:
\begin{eqnarray}\label{eq1}
F(q_\a(x) \mid x) = \a.
\end{eqnarray}
%\noindent \textcolor{red}{To be completed after: add details about articles studying quantiles for both dependent and independent data with censorship assumption or without}

\subsection{A nonparametric estimator of conditional quantiles}
It is clear that an estimator of $q_\a(x)$ can easily be deduced from an estimator of $F(t\mid x)$. Let us recall that in the case of {\it complete data}, a well-known kernel-estimator of the conditional distribution function is given by
\begin{eqnarray}\label{estF1}
F_n(t \mid x) =  \sum_{i=1}^n W_{n,i}(x) \; H(h_{n,H}^{-1}(t-T_i)),
\end{eqnarray}
where 
\begin{eqnarray}\label{NWW}
W_{n,i}(x) = \frac{K\left(h _{n,K}^{-1} d(x, X_i)\right) }{\sum_{i=1}^n K\left(h _{n,K}^{-1} d(x, X_i)\right)},
\end{eqnarray}
 are the well-known Nadaraya-Watson weights. Here $K$ is a real-valued kernel function, $H$ a cumulative distribution function and $h_K:= h_{n,K}$ (resp. $h_H:= h_{n,H}$) a sequence of positive real numbers which decreases to zero as $n$ tends to infinity. This estimator given by (\ref{estF1}) has been introduced in \cite{ferraty2006} in the general setting.
 
An appropriate estimator of the conditional distribution function $F(t\mid x)$ for censored data is then obtained by adapting (\ref{NWW}) in order to put more emphasis on large values of the interest random variable $T$ which are more censored than small one. Based on the same idea as in \cite{carbonez} and \cite{khardani}, we consider the following weights
\begin{eqnarray}\label{NWWC}
\widetilde{W}_{n,i}(x) = \frac{1}{h_K} K\left(h _K^{-1} d(x, X_i)\right) \frac{\delta_i}{\overline{G}(Y_i) \sum_{i=1}^nh_K^{-1} K\left(h _K^{-1} d(x, X_i)\right)},
\end{eqnarray}
where $\overline{G}(\cdot) = 1- G(\cdot)$.
Now, we consider a {\it "pseudo-estimator''} of $F(t\mid x)$ given by:
\begin{eqnarray}\label{pseudo}
\widetilde{F}_n(t\mid x) = \frac{\sum_{i=1}^n \delta_i \bar{G}^{-1}(Y_i) K\left( h_K^{-1} d(x,X_i)\right)\; H(h_{H}^{-1}(t-Y_i))}{\sum_{i=1}^n K\left( h_K^{-1} d(x,X_i)\right)} := \frac{\widetilde{F}_{n}(x,t)}{\ell_n(x)},
\end{eqnarray}

where 
\begin{eqnarray*}
\widetilde{F}_{n}(x,t) = \frac{1}{n\E(\Delta_1(x))} \sum_{i=1}^n \delta_i \bar{G}^{-1}(Y_i) \; H\left(h_H^{-1}(t-Y_i)\right) \Delta_i(x),
\end{eqnarray*}
and
$$
\ell_n(x) = \frac{1}{n\E(\Delta_1(x))} \sum_{i=1}^n \Delta_i(x),
$$
%and
%\begin{eqnarray*}
%\ell_n(x) = \frac{1}{n\E(\Delta_1(x))} \sum_{i=1}^n \Delta_i(x)
%\end{eqnarray*}
where $\Delta_i(x) = K\left(d(x,X_i)/h_K \right)$.
In practice $G$ is unknown, we use the \cite{kaplan} estimator of $G$ given by:

$$\overline{G}_n(t) = \left\lbrace 
\begin{array}{lcl} 
\prod_{i=1}^n \left( 1-\frac{1-\delta_{(i)}}{n-i+1}\right)^{\1_{\left\{ Y_{(i)} \leq t \right\}}} & & \mbox{if}\quad t<Y_{(n)},\\ 
0 & & \mbox{Otherwise},\\ 
\end{array}\right.$$ 
where $Y_{(1)} < Y_{(2)} < \dots < Y_{(n)}$ are the order statistics of $(Y_i)_{1\leq i \leq n}$ and $\delta_{(i)}$ is the concomitant of $Y_{(i)}$. Therefore, the estimator of $F(t\mid x)$ is given by:
\begin{eqnarray}\label{estF2}
\widehat{F}_n(t\mid x) = \frac{\widehat{F}_{n}(x,t)}{\ell_n(x)},
\end{eqnarray}
where
$$
\widehat{F}_{n}(x,t) = \frac{1}{n\E(\Delta_1(x))} \sum_{i=1}^n \delta_i \bar{G}_n^{-1}(Y_i) \; H(h_{H}^{-1}(t-Y_i)) \Delta_i(x).
$$
Then a natural estimator of $q_\a(x)$ is given by:
\begin{eqnarray}\label{estQ}
\widehat{q}_{n,\a}(x) = \inf\{ y : \widehat{F}_n(y\mid x) \geq \a \},
\end{eqnarray}
which satisfies:
\begin{eqnarray}\label{verif}
\widehat{F}_n(\widehat{q}_{n,\a}(x) \mid x) = \a.
\end{eqnarray}

\section{Assumptions and main results}\label{S3}
In order to state our results, we introduce some notations. Let $\F_i$ be the $\sigma$-field generated by $\left( (X_1, T_1), \dots, (X_i, T_i) \right)$ and $\mathcal{G}_i$ the one generated by $\left( (X_1, T_1), \dots, (X_i, T_i), X_{i+1} \right).$ Let $B(x,u)$ be the ball centered at $x\in E$ with radius $u$. Let $D_i(x) := d(x,X_i)$ so that $D_i(x)$ is a nonnegative real-valued random variable. Working on the probability space $\left(\Omega, \mathcal{A}, \P \right)$, let $F_x(u) = \P\left(D_i(x) \leq u \right) := \P\left(X_i \in B(x,u)\right)$ and $F_x^{\F_{i-1}}(u) = \P\left( D_i(x) \leq u\mid \F_{i-1}\right) = \P\left(X_i\in B(x,u)\mid \F_{i-1} \right)$ be the distribution function and the conditional distribution function, given the $\sigma$-field $\F_{i-1}$, of $(D_i(x))_{i\geq 1}$ respectively. Denote by $o_{a.s.}(u)$ a real random function $\ell$ such that $\ell(u)/u$ converges to zero almost surely as $u\rightarrow 0.$ Similarly, define $O_{a.s.}(u)$ a real random function $\ell$ such that $\ell(u)/u$ is almost surely bounded. Furthermore, for any distribution function $L$, let $\tau_L = \sup\{ t, \mbox{such that}\; L(t) <1 \}$ be the support's right endpoint. Let $S$ be a compact set such that $q_\a(x)\in S \cup (-\infty, \tau]$, where $\tau < \tau_G \wedge \tau_F.$ 
%Furthermore, let $S$ be a compact set of $\ER$ such that $q_\a(x) \in \overset{\circ}{S}$, where $\overset{\circ}{S}$ denotes the interior of $S.$
\subsection{Rate of strong consistency}

Our results are stated under some assumptions we gather hereafter for easy reference.

\begin{itemize}
\item[(A1)] $K$ is a nonnegative bounded kernel of class $\mathcal{C}^1$ over its support $[0,1]$ such that $K(1) >0$. The derivative $K^\prime$ exists on $[0,1]$ and satisfy the condition $K^\prime(t) <0,$ for all $t\in [0,1]$ and $|\int_0^1 (K^j)^\prime(t) dt| < \infty$ for $j=1,2.$ 
\item[(A2)] For $x\in E$, there exists a sequence of nonnegative random functionals $(f_{i,1})_{i\geq 1}$ almost surely bounded by a sequence of deterministic quantities $(b_i(x))_{i\geq 1}$ accordingly, a sequence of random functions $(g_{i,x})_{i\geq 1}$, a deterministic nonnegative bounded functional $f_1$ and a nonnegative real function $\phi$ tending to zero, as its argument tends to 0, such that 
\begin{itemize}
\item[(i)] $F_x(h) = \phi(h)f_1(x) + o(\phi(h))$ as $h\rightarrow 0.$
\item[(ii)] For any $i\in \mathbb{N}$, $F_x^{\mathcal{F}_{i-1}}(h) = \phi(h) f_{i,1}(x) + g_{i,x}(h)$ with $g_{i,x}(h) = o_{a.s.}(\phi(h))$ as $h\rightarrow 0,$ $g_{i,x}(h)/\phi(h)$ almost surely bounded and $n^{-1}\sum_{i=1}^ng_{i,x}^j(h)=o_{a.s.}(\phi^j(h))$ as $n\rightarrow \infty$, $j=1,2.$
\item[(iii)] $n^{-1}\sum_{i=1}^nf_{i,1}^j(x) \rightarrow f_1^j(x),$ almost surely as $n\rightarrow\infty$, for $j=1,2.$
\item[(iv)] There exists a nondecreasing bounded function $\tau_0$ such that, uniformly in $s\in[0,1]$,
$\frac{\phi(hs)}{\phi(h)}=\tau_0(s) +o(1)$, as $h \downarrow 0$ and, for $j \geq 1$, $\int_0^1 (K^j(t))' \tau_0(t) dt < \infty.$
\item[(v)] $n^{-1}\sum_{i=1}^n b_i(x) \rightarrow D(x) <\infty$ as $n\rightarrow\infty.$
\end{itemize}
\item[(A3)] The conditional distribution function $F(t\mid x)$ has a positive first derivative with respect to $t$, for all $x\in E$, denoted $f(t\mid x)$ and satisfies 
\end{itemize}
\begin{itemize}
\item[(ii)] $\int_{\mathbb{R}}|t| f(t\mid x) dt <\infty$, for all $x\in E$,
\item[(ii)] For any $x\in E$, there exist $V(x)$ a neighborhood of $x$, some constants $C_x>0$, $\beta >0$ and $\nu >0$, such that for $j=0,1$, we have $\forall (t_1,t_2) \in S\times S$, $\forall (x_1, x_2)\in V(x)\times V(x)$, $$\Big|F^{(j)}(t_1\mid x_1) - F^{(j)}(t_2 \mid x_2)\Big|\leq C_x \left( d(x_1,x_2)^\beta + |t_1 - t_2|^\nu\right).$$
\end{itemize}
\begin{itemize}
\item[(A4)] For any $m\geq 1$ and $j=0,1$, $\E\left[\left(H^{(j)}(h_{H}^{-1}(t-T_i))\right)^m \mid \G_{i-1}\right] = \E\left[\left(H^{(j)}(h_{H}^{-1}(t-T_i))\right)^m \mid X_i\right] $ 
\item[(A5)] The distribution function $H$ has a first derivative $H^{(1)}$ which is positive and bounded and satisfies $\int |u|^\nu H^{(1)}(u) du < \infty.$
\item[(A6)] For any $x^\prime\in E$ and $m\geq 2$, $\sup_{t\in\S}|g_{m}(x^\prime,t)| := \sup_{t\in\S}|\E[H^m(h_{H}^{-1}(t-T_1)) \mid X_1=x^\prime]| < \infty$ and $g_{m}(x^\prime,t)$ is continuous in $V(x)$ uniformly in $t$:
$$
\sup_{t\in\S}\sup_{x^\prime\in B(x,h)}|g_{m}(x^\prime,t) - g_{m}(x,t)| = o(1).
$$
\item[(A7)]$(C_n)_{n \geq 1}$and $(T_n,X_n)_{n \geq 1}$ are  independent.

\end{itemize}
\noindent {\it Comments on hypothesis}: Conditions (A1) involves the ergodic nature of the data and the small ball techniques used in this paper. Several examples where condition (A1)(ii) is satisfied are discussed in \cite{laib2011}. Assumption (A3)(ii) involves the conditional probability and conditional probability density, it means that $F(\cdot\mid \cdot)$ and $f(\cdot\mid \cdot)$ are continuous with respect to each variable. Assumption (A4) is of Markov's nature. Hypothesis (A1) and (A5) impose some regularity conditions upon the kernels used in our estimates. Condition (A6) stands as regularity condition that is of usual nature. 
 The independence assumption between $(C_i)_i$ and $(X_i,T_i)_i$, given by (A7), may seem to
be strong and one can think of replacing it by a classical conditional independence
assumption between $(C_i)_i$ and $(T_i)_i$ given $(X_i)_i$. However considering the latter
demands an a priori work of deriving the rate of convergence
of the censoring variable's conditional law (see \cite{deheuvels}). Moreover
our framework is classical and was considered by \cite{carbonez} and
\cite{kohler}, among others.

\begin{proposition}\label{convFhat}
Assume that conditions (A1)-(A7) hold true and that
\begin{eqnarray}\label{cond}
n\phi(h_K) \rightarrow \infty \quad \mbox{and}\quad \frac{\log n}{n\phi(h_K)}\rightarrow 0 \quad \mbox{as}\quad n\rightarrow\infty.
\end{eqnarray}
Then, we have
$$
\sup_{t \in S} \Bigl\lvert \widehat{F}_n(t \mid x) - F(t \mid x) \Bigl\rvert = O_{a.s.}\left(h_K^\beta + h_H^\nu \right) + O_{a.s.}\left( \sqrt{\frac{\log n}{n\phi(h_K)}}\right).
$$
\end{proposition}

\begin{theorem}\label{theoremQ}
Under the same assumptions of Proposition \ref{convFhat}, we have
\begin{eqnarray}
\Big| \widehat{q}_{n,\a}(x)-  q_\a(x) \Big| = O_{a.s.}\left(h_K^\beta + h_H^\nu \right) + O_{a.s.}\left( \sqrt{\frac{\log n}{n\phi(h_K)}}\right).
\end{eqnarray}
\end{theorem}

\subsection{Asymptotic normality}
The aim of this section is to establish the asymptotic normality which induces a confidence interval of the conditional quantiles estimator. For that purpose we need to introduce further notations and assumptions. We assume, for $k=1,2$, that $\E\left( |\delta_1 \bar{G}^{-1}(Y_1)H(h_H^{-1}(t - Y_1))|^k\right) < \infty$ and that, for a fixed $x \in E$, the conditional variance, of $\delta_1 \bar{G}^{-1}(Y_1)H(h_H^{-1}(t - Y_1))$ given $X_1=x$, say, \\ $W_2(t\mid x) := \E\left[ \left(\delta_1 \bar{G}^{-1}(Y_1)H(h_H^{-1}(t - Y_1)) - F(t\mid x) \right)^2\mid X_1=x\right]$  exists.
\begin{itemize}
\item[(A8)] 
\begin{itemize}
\item[(i)] The conditional variance of $\delta_i \bar{G}^{-1}(Y_i) H(h_H^{-1}(t - Y_i))$ given the $\sigma$-field $\G_{i-1}$ depends only on $X_i$, i.e., for any $i\geq 1$, $\E\left[\left( \delta_i \bar{G}^{-1}(Y_i) H(h_H^{-1}(t - Y_i)) - F(t\mid X_i)\right)^2 \mid \G_{i-1}\right] = W_2(t\mid X_i)$ almost surely. %Moreover, the function $W_2(\cdot\mid \cdot)$ is continuous in a neighborhood of $x$, that is,
%$$
%\sup_{x^\prime \in B(x, h)}\Big| W_2(t\mid x^\prime) - W_2(t\mid x) \Big| = o(1) \quad \mbox{as}\;\; h\rightarrow 0.
%$$
\item[(ii)] For some $\delta>0$, $\E[|\delta_1 \bar{G}^{-1}(Y_1)H(h_H^{-1}(t - Y_1))|^{2+\delta}]<\infty$ and the function \\$\overline{W}_{2+\delta}(t\mid u):=\E(|\delta_1 \bar{G}^{-1}(Y_i)H(h_H^{-1}(t - Y_i)) - F(t\mid x)|^{2+\delta}\mid X_i=u)$, $u\in E$, is continuous in a neighborhood of $x.$\\
\end{itemize}
\item[(A9)] The distribution function of the censored random variable, $G$ has bounded first derivative $G^{(1)}.$
\end{itemize}

\begin{theorem}\label{normF}
Assume that assumptions (A1)-(A9) hold true and condition (\ref{cond}) is satisfied, then we have
\begin{eqnarray*}
\sqrt{n\phi(h_K)} \left(\widehat{F}_n(t\mid x)-  F(t\mid x)  \right)  \stackrel{\mathcal{D}}{\longrightarrow} \mathcal{N}\left( 0, \sigma^2(x,t)\right),
\end{eqnarray*}
where $ \stackrel{\mathcal{D}}{\longrightarrow}$ denotes the convergence in distribution and 
\begin{eqnarray*}\sigma^2(x,t)= \frac{M_2}{M_1^2}\frac{F(t \mid x)\left(\bar{G}^{-1}(t) - F(t \mid x) \right)}{f_1(x)},
\end{eqnarray*}
where $M_j = K^j(1) - \int_0^1 (K^j)^\prime \tau_0(u) du.$
\end{theorem}

\begin{theorem}\label{normality}
Under the same assumptions and conditions of Theorem \ref{normF}, we have
\begin{eqnarray*}
\sqrt{n\phi(h_K)} \left(\widehat{q}_{n,\a}(x)-  q_\a(x)  \right)  \stackrel{\mathcal{D}}{\longrightarrow} \mathcal{N}\left( 0, \gamma^2(x, q_\a(x))\right),\end{eqnarray*}
\begin{eqnarray*}
\gamma^2(x,q_\a(x)) = \frac{\sigma^2(x, q_\a(x))}{f^2(q_\a(x)\mid x))} = \frac{M_2}{M_1^2 f_1(x)}\frac{\a \left[ \bar{G}^{-1}(q_{\a}(x)) - \a \right]}{f^2(q_\a(x)\mid x))}.
\end{eqnarray*}
\end{theorem}
Observe that in Theorem \ref{normality} the limiting variance, for a given $x\in E$, $\gamma^2(\cdot\mid x)$ contains the unknown function $f_1(\cdot)$, the normalization depends on the function $\phi(\cdot)$ which is not identifiable explicitly and the theoretical conditional quantile $q_\a(x)$. Moreover, we have to estimate the quantities $f(\cdot\mid x)$, $\tau_0(\cdot)$ and $G(\cdot)$. The corollary below allows us to get a suitable form of Central Limit Theorem which can be used in practice to estimate interval prediction.
First of all, we give an estimator of each unknown quantity in Theorem \ref{normality}. To estimate, for a fixed $x\in E$, $\gamma(\cdot\mid x)$ the quantities $\bar{G}(\cdot)$, $F(\cdot\mid x)$ and $q_\a(x)$ should be replaced by their estimators $\bar{G}_n(\cdot)$, $\widehat{F}_n(\cdot\mid x)$ and $\widehat{q}_{n,\a}(x)$ respectively. 

%Then, we get
%\begin{eqnarray}
%W_{2,n}(\widehat{q}_{n,\a}(x)\mid x) &=& \frac{\sum_{i=1}^n  \frac{\delta_i}{\bar{G}_n^2(Y_i)} H^2\left( \frac{\widehat{q}_{n,\a}(x)-Y_i}{h_H}\right) K\left(\frac{d(x,X_i)}{h_K} \right)}{\sum_{i=1}^n K\left(\frac{d(x,X_i)}{h_K} \right)} - \a^2.
%\end{eqnarray}
Now, using the decomposition given by assumption (A2)(i), one can estimate $\tau_0(u)$ by $\tau_n(u) = F_{x,n}(uh)/F_{x,n}(h)$, where $F_{x,n}(u) = 1/n \sum_{i=1}^n \1_{\{ d(x,X_i)\leq u\}}.$
Therefore, for a given kernel $K$, an estimators of $M_1$ and $M_2$, namely $M_{1,n}$ and $M_{2,n}$ respectively, are obtained by plug-in $\tau_n$, in place of $\tau_0$, in their respective expressions.
\subsection{Application to predictive interval}
\begin{corollary}\label{corr}
Assume that conditions (A1)-(A9) hold true, $K^\prime$ and $(K^2)^\prime$ are integrable functions and 
\begin{eqnarray*}
\frac{M_{1,n} \;\hat{f}_n(\widehat{q}_{\a,n}(x)\mid x)}{\sqrt{M_{2,n}}} \sqrt{\frac{n F_{x,n}(h_K)}{\a \left( \bar{G}^{-1}_n(\widehat{q}_{\a,n}(x)) - \a\right)}} \; \left(\widehat{q}_{n,\a}( x) - q_{\a}(x)\right) \stackrel{\mathcal{D}}{\longrightarrow} \mathcal{N}( 0, 1),
\end{eqnarray*}
where $\hat{f}_n( \cdot \mid x)$ is an estimator of the conditional density function $f(\cdot\mid x).$
\end{corollary}

The corollary \ref{corr} can be now used to provide the $100(1-\a)\%$ confidence bands for $q_\a(x)$ which is given, for $x\in E$, by
\begin{eqnarray*}
\widehat{q}_{n,\a}(x) \pm c_{\a/2} \frac{M_{1,n} \;\hat{f}_n(\widehat{q}_{\a,n}(x)\mid x)}{\sqrt{M_{2,n}}} \sqrt{\frac{\a \left( \bar{G}^{-1}_n(\widehat{q}_{\a,n}(x)) - \a\right)}{n F_{x,n}(h_K)}}.
\end{eqnarray*}
where $c_{\a/2}$ is the upper $\a/2$ quantile of the distribution of ${\cal N}(0,1).$
In the following section we give an application of this corollary for interval prediction of the daily electricity peak demand under random censorship.
 
\section{Interval prediction of peak load with censored smart meter data}\label{S4}
The evolution of peak electricity demand can be considered an important system design metric for grid operators and planners. In fact, overall electricity demand and the daily schedules of electricity usage are currently major influences on peak load. In the future, peak load will be influenced by new or increased demands such as the penetration of clean energy technologies (wind and photovoltaic generation), such as electric vehicles and the increased use of electricity for heating and cooling through technologies such as heat pumps. Smart grid technology, through the use of advanced monitoring and control equipment, could reduce peak demand and thus prolong and optimize use of the existing infrastructure. Regularly, the electricity network constraints are evaluated on a specific area in order to prevent over-voltage problems on the grid. Very localized peak demand forecasting is needed to detect voltage constraints on each node and each feeder in the Low-Voltage network.
\begin{figure}
\begin{center}
\includegraphics[height=9cm,width=15cm]{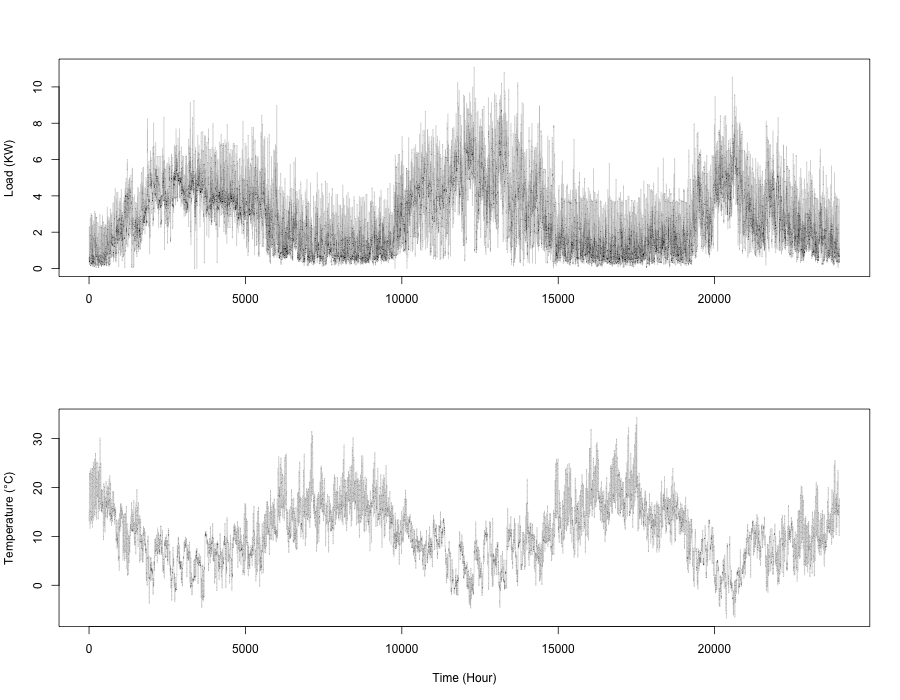}
\end{center}
\caption{Hourly measurements for electricity demand and temperature for 1000 days.}
\label{load}
\end{figure}

Peak demand forecasting of aggregated electricity demand has been widely studied in statistical literature and several approaches have been proposed to solve this issue, see for instance, \cite{sigauke} and \cite{goia} for short-term peak forecasting and \cite{hyndman} for long-term density peak forecasting. The arrival of  Automated Meter Reading (AMR) allows us to receive energy demand measurement at a finite number of equidistant time points, e.g. every half-hour or every ten minutes. As far as we know, nothing has been done for household-level peak demand forecasting. In this section we are interested in the estimation of interval prediction of peak demand at the customer level. For a fixed day $d$, let us denote by $(\mathcal{L}_d(t_j))_{j=1, \dots, 24}$ the hourly measurements sent by the AMR of some specific customer. The peak demand observed for the day $d$ is defined as
\begin{eqnarray*}
\mathcal{P}_d = \max_{j=1,\dots, 24} \mathcal{L}_d(t_j).
\end{eqnarray*}
The transmission of the consumed energy from the AMR to the system information might be made, for instance, by wireless technology. Unfortunately, in practice we cannot receive on time the hole measurements for every day. In fact many sources of such kind of censorship could arise, for instance an interruption in the wireless communication during the day or a physical problem with the AMR. Whenever we receive the hole data one can determine the peak for that day, otherwise (when data are censored) we cannot calculate the true peak. In such case one can delete that observation from the sample which leads to reduce the available information. In this paper we suggest to keep censored observations in our sample and use them to predict peak demand intervals.

It is well-known that peak demand is very correlated with temperature measurments. Figure \ref{load} shows the hourly measurements of electricity demand  and temperature during 1000 days. One can easily observe a seasonality in the load curve which reflects the sensitivity of energy consumption, for that customer, to weather conditions. Figure \ref{peak} provides a sample of 10 curves of houly temperature measures and the associated electricity demand curves. Observed peak, for each day, is plotted in solid circles. 
\begin{figure}
\begin{center}
\includegraphics[height=6.5cm,width=15cm]{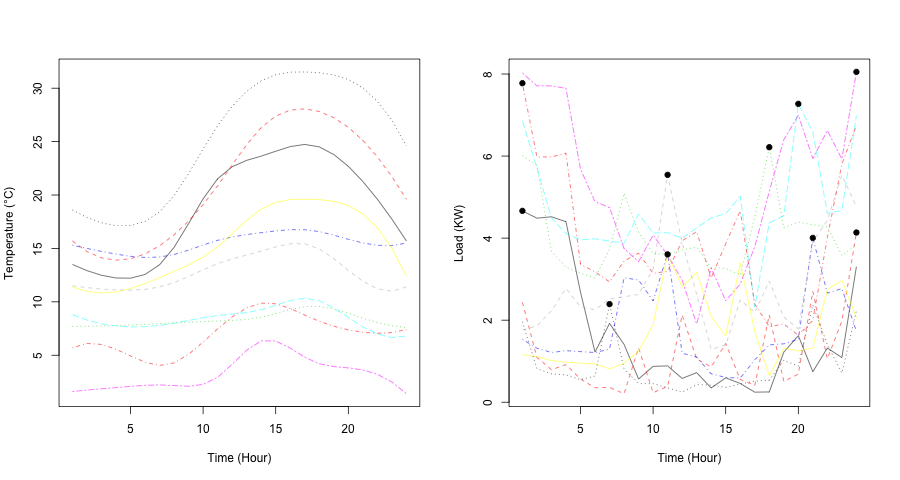}
\end{center}
\caption{A sample of 10 daily temperature curves and the associated electricity demand curves. Observed daily peaks are in solid circle.}
\label{peak}
\end{figure}
We split our sample of 1000 days into learning sample containing the first 970 days and a testing sample with the last 30 days. From the learning sample we selected 30\% of days within which we generated randomly the censorship. Figure \ref{samplecens} provides a sample of 6 censored daily load curves. For those days, the AMR send hourly electricity consumption until a certain time $t_c\in [1,24]$ which corresponds to the time of censorship which is plotted in dashed line in Figure \ref{samplecens}. For a censored day, we define the censored random variable 
\begin{eqnarray*}
C_d = \max_{j=1, \dots t_c} \mathcal{L}_d(t_j),
\end{eqnarray*}
where $t_c$ is the time from which we don't receive data from the smart meter.
\begin{figure}
\begin{center}
\includegraphics[height=10cm,width=15cm]{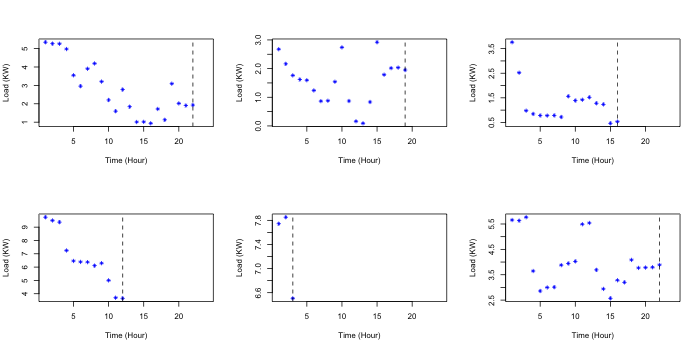}
\end{center}
\caption{A sample of 6 censored daily load curves. Observed values of electricity consumption are plotted in star points, dashed line corresponds to the time of censorship for each day.}
\label{samplecens}
\end{figure}
Therefore, our sample is formed as follow $(X_d, Y_d)_{d=1, \dots, 970}$, where $X_d$ is the predicted temperature curve for the day $d$ and $Y_d = \mathcal{P}_d$ for completely observed days and $Y_d = C_d$ for censored ones. Here, we investigate, for each day $d=971, \dots, 1000$, the conditional quantile functions of $Y_d$ given the predicted temperature curve $X_d$. The $5\%$ and $95\%$ quantiles consists of the $90\%$ confidence intervals of the last 30 peak load in the testing sample, say $[q_{0.05}(X_d), q_{0.95}(X_d)]$ for $d=971, \dots, 1000.$ Note that these confidence intervals are derived directly from the conditional quantile functions given by (\ref{estQ}). To estimate conditional quantiles we chose the quadratic kernel defined by $K(u) = 1.5(1-u^2)\1_{[0,1]}$. Because the daily temperature curves are very smooth, we chosed as semi-metric $d(\cdot, \cdot)$ the $L_2$ distance between the sec ond derivative of the curves. Finally, we considered the optimal bandwidth $h:= h_K=h_H$ chosen by the cross-validation method on the $k$-nearest neighbors (see \cite{ferraty2006}, p.102 for more details). Figure \ref{intervals} provides our results for the peak load interval prediction for the testing sample. The true peaks are plotted in solid triangles. Solid circles represent the conditional median values. On can easily observe that the conditional median is a consistent predictor of the peak. In fact, let us define the Mean Absolute Prediction Error as
\begin{eqnarray*}
\mbox{MAPE} = \frac{1}{30}\sum_{d=1}^{30} \frac{|\mathcal{P}_d - \hat{q}_{0.5}(X_d)|}{\mathcal{P}_d},
\end{eqnarray*}
where $\mathcal{P}_d$ is the true value of the peak for the day $d$ and $\hat{q}_{0.5}(X_d)$ its predicted value based on the conditional median. We obtain here $\mbox{MAPE}=0.24.$ Observe that we over-estimate the peak of the 16th day.
\begin{figure}
\begin{center}
\includegraphics[height=7.5cm,width=14cm]{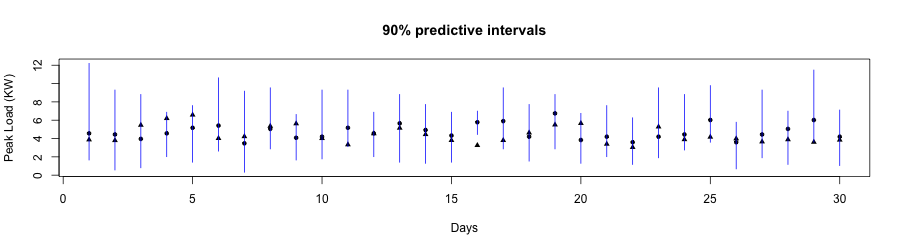}
\end{center}
\caption{$90\%$ predictive intervals of the peak demand for the last 30 days.}
\label{intervals}
\end{figure}

\section*{Acknowledgment}

The first author would like to thank Scottish and Southern Power Distribution SSEPD for support and funding via the New Thames Valley Vision Project (SSET203 - New Thames Valley Vision) funded through the Low Carbon Network Fund.

\section{Proofs of main results}\label{S5}
In order to proof our results, we introduce some further notations. Let 
$$
\overline{\widetilde{F}}_n(x,t) = \frac{1}{n\E(\Delta_1(x))} \sum_{i=1}^n \E\left[ \delta_i \bar{G}^{-1}(Y_i) \; H(h_{H}^{-1}(t-Y_i)) \; \Delta_i(x) \mid \F_{i-1}\right]
$$
and
$$
\overline{\ell}_n(x) = \frac{1}{n\E(\Delta_1(x))} \sum_{i=1}^n \E\left[ \Delta_i(x) \mid \F_{i-1}\right].
$$

Now, lets introduce the decomposition hereafter. For $x\in E$, set
%\begin{eqnarray}
% \widehat{F}_n(t \mid x) - F(t \mid x) =  \frac{\widehat{F}_{n,2}(x) - \widetilde{F}_{n,2}(x)}{\widehat{F}_{n,1}(x)} +\frac{\widetilde{F}_{n,2}(x) -F(t\mid x) }{\widehat{F}_{n,1}(x)} + \frac{F(t\mid x)}{\widehat{F}_{n,1}(x)}\left\{ 1-\widehat{F}_{n,1}(x)\right\}
 %&+&\frac{\E {\widetilde{F}_n(t , x))-{F}(t , x)} }{\ell_n(x) }+{F}(t , x)\frac{1-\ell_n(x)}{\ell_n(x)} 
%\end{eqnarray}

\begin{eqnarray}\label{decomp}
 \widehat{F}_n(t \mid x) - F(t \mid x) =  \widehat{F}_n(t\mid x) - \widetilde{F}_n(t\mid x) +  \widetilde{F}_n(t\mid x) -F(t\mid x).
 %&+&\frac{\E {\widetilde{F}_n(t , x))-{F}(t , x)} }{\ell_n(x) }+{F}(t , x)\frac{1-\ell_n(x)}{\ell_n(x)} 
\end{eqnarray}
To get the proof of Proposition \ref{convFhat}, we establish the following Lemmas.

\begin{definition}
A sequence of random variables $(Z_n)_{n\geq1}$ is said to be a sequence of martingale differences with respect to the sequence of $\sigma$-fields $(\mathcal{F}_n)_{n\geq1}$ whenever $Z_n$ is $\mathcal{F}_n$ measurable and $\mathbb{E}\left(Z_n\mid\mathcal{F}_{n-1} \right)=0$ almost surely.
\end{definition}
In this paper we need an exponential inequality for partial sums of unbounded martingale differences that we use to derive asymptotic results for the Nadaraya-Watson-type multivariate quantile regression function estimate built upon functional ergodic data. This inequality is given in the following lemma.
\begin{lemma}\label{fact1}
Let $(Z_n)_{n\geq1}$ be a sequence of real martingale differences with respect to the sequence of $\sigma$-fields $(\mathcal{F}_n=\sigma(Z_1,\dots,Z_n))_{n\geq1}$, where $\sigma(Z_1, \dots, Z_n)$ is the $\sigma$-filed generated by the random variables $Z_1, \dots, Z_n$. Set $S_n=\sum_{i=1}^n Z_i.$ For any $p\geq2$ and any $n\geq1$, assume that there exist some nonnegative constants $C$ and $d_n$ such that
\begin{eqnarray}
\mathbb{E}\left( Z_n^p\mid\mathcal{F}_{n-1}\right) \leq C^{p-2}p!d_n^2\quad\quad\mbox{almost surely.}
\end{eqnarray}
Then, for any $\epsilon>0$, we have 
$$
\mathbb{P}\left(|S_n|>\epsilon \right)\leq 2\exp\left\{-\frac{\epsilon^2}{2(D_n+C\epsilon)} \right\},
$$
where $D_n=\sum_{i=1}^n d_i^2.$
\end{lemma}
As mentioned in \cite{laib2011} the proof of this lemma follows as a particular case of Theorem 8.2.2 due to \cite{gine}.

We consider also the following technical lemma whose proof my be found in \cite{laib2010}.

\begin{lemma}\label{lem1}
Assume that assumptions (A1) and (A2)(i), (A2)(ii) and (A2)(iv) hold true. For any real numbers $1\leq j\leq 2+\delta$ and $1\leq k\leq 2+\delta$ with $\delta >0,$ as $n\rightarrow\infty$, we have
\begin{itemize}
\item[(i)] $\displaystyle\frac{1}{\phi(h_K)}\mathbb{E}\left[ \Delta_i^j(x)\mid \mathcal{F}_{i-1}\right] = M_j f_{i,1}(x) + \mathcal{O}_{a.s.}\left( \frac{g_{i,x}(h_K)}{\phi(h_K)}\right),$
\item[(ii)] $\displaystyle\frac{1}{\phi(h_K)}\mathbb{E}\left[ \Delta_i^j(x)\right] = M_j f_{1}(x) + o(1)$,
\item[(iii)] $\displaystyle\frac{1}{\phi^k(h_K)}\left( \mathbb{E}(\Delta_1(x))\right)^k = M_1^k f_1^k(x) + o(1).$
\end{itemize}
\end{lemma}
\begin{lemma}\label{lem2}
Assume that hypotheses (A1)-(A2) and the condition (\ref{cond}) are satisfied. Then, for any $x\in E$, we have

\begin{itemize}
\item[(i)] $\ell_n(x) - \overline{\ell}_n(x) = O_{a.s.}\left( \sqrt{\displaystyle\frac{\log n}{n\phi(h_K)}}\right)$,
\item[(ii)] $\lim_{n\rightarrow\infty} \ell_n(x) = \lim_{n\rightarrow\infty} \overline{\ell}_n(x) = 1\quad\mbox{a.s.}.$
\end{itemize}

\end{lemma}
\begin{proof} See the proof of Lemma 3 in \cite{laib2010}.
\end{proof}

%%%%%%%%%%%%%%%%%%%%%
%%%%%%%%%%%%%%%%%%%%%

%\begin{Lemma}\label{gamman}
%Under Assumptions \textbf{A1-A3}
%\begin{eqnarray*}
%|\ell_n(x)-1|=O\left(\sqrt{\frac{\log n}{n\phi_x(h)}}\right) \ \ {a.co.}
%\end{eqnarray*}
%\end{Lemma}
%For the proof of Lemma \ref{gamman}, see the demonstration of equality (24), page 12 in Ferraty (2008).

%%%%%

%\begin{Lemma}\label{bias}
% Under Assumptions ????  we have
% \begin{eqnarray*}
%\sup_{t\in S} \Big|\E\left[ \widetilde{F}_n(x))\right] -F(t \mid x) \Big|=O\left(h^\beta + h^\nu \right)\quad \mbox{a.s.}.
%\end{eqnarray*}
%\end{Lemma}

%%%%%

%\begin{Lemma}\label{Var}
%Under Assumptions ???, we have
%\begin{eqnarray*}
%\sup_{t\in S }
%\Big |\widetilde{F}_n(t , x)) - \E \left[\widetilde{F}_n(t , x))\right]
%\Big|= O\left\{\left(\frac{\log n}{n \phi_x(h)}\right)^{1/2}\right\} \quad\quad
%\mbox{a.s.}
%\label{l2}
%\end{eqnarray*}
%\end{Lemma}

\begin{proof} {\it of Proposition \ref{convFhat}}

\noindent Making use of the decomposition (\ref{decomp}), the result follows as a direct consequence of Lemmas \ref{diff} and \ref{equiv} below. 

\begin{lemma}\label{diff}
Under Assumptions (A1)-(A7) and the condition (\ref{cond}) , we have
\begin{eqnarray*}
\sup_{t\in S } \Big|\widetilde{F}_n(t \mid x) - F(t \mid x) \Big|=O_{a.s.}(h_K^\beta + h_H^\nu) + O_{a.s.}\left(\sqrt{\frac{\log n}{n\phi(h_K)}}\right).
\end{eqnarray*}
\label{lil}
\end{lemma}

%%%%%%%

\begin{lemma}\label{equiv}
Assume that hypothesis (A1)-(A7) and the condition (\ref{cond}) hold, we have
\begin{eqnarray*}
\sup_{t\in S } \Big|\widehat{F}_n(t \mid x) - \widetilde{F}_n(t \mid x) \Big|=O_{a.s.}\left(\sqrt{\frac{\log\log n}{n}}\right).
\end{eqnarray*}
\label{lil}
\end{lemma}
\end{proof}
%%%%%%%%
We provide, in the following lemma, the almost sure consistency, without rate, of $\widehat{q}_{n,\a}(x)$.
\begin{lemma}\label{convNR}
Under assumptions of Proposition \ref{convFhat}, we have
$$
\lim_{n\rightarrow\infty}\Big| \widehat{q}_{n,\a}(x) - q_{\a}(x)\Big| = 0, \quad \mbox{a.s.}
$$
\end{lemma}
\begin{proof}{\it of Lemma \ref{convNR}}

\noindent Following the similar steps as in \cite{ezzahrioui}, the proof of this lemma is based in the following decomposition. As $F(\cdot\mid x)$ is a distribution function with a unique quantile of order $\a$, then for any $\epsilon > 0$, let:
$$
\eta(\epsilon) = \min\{F(q_\a(x) + \epsilon \mid x) - F(q_\a(x) \mid x), F(q_\a(x) \mid x) - F(q_\a(x) - \epsilon\mid x) \},
$$ 
then
$$
\forall \epsilon >0, \forall t >0, \quad|q_\a(x) - t| \geq \epsilon \Rightarrow |F(q_\a(x)\mid x) - F(t\mid x)| \geq \eta(\epsilon).
$$
Now, using (\ref{eq1}) and (\ref{verif}) we have
\begin{eqnarray}\label{ineq}
\Big| F(\widehat{q}_{n,\a}(x)\mid x) - F(q_{\a}(x)\mid x)\Big| &=& \Big| F(\widehat{q}_{n,\a}(x) \mid x) - \widehat{F}_n(\widehat{q}_{n,\a}(x)\mid x)\Big|\nonumber\\
&\leq& \sup_{t\in\S}\Big| F(t \mid x) - \widehat{F}_n(t \mid x)\Big|
\end{eqnarray}
The consistency of $\widehat{q}_{n,\a}$ follows then immediately from Proposition \ref{convFhat}, the continuity of $F(\cdot\mid x)$ and the following inequality
$$
\sum_n^\infty \P\left(|\widehat{q}_{n,\a}(x) - q_\a(x)| \geq \epsilon \right) \leq \sum_n^\infty \P\left( \sup_{t\in\S} |F(t \mid x) - \widehat{F}_n(t \mid x)|\geq \eta(\epsilon)\right).
$$
\end{proof}
\begin{proposition}\label{density}
Under assumptions (A1)-(A5) together with condition (\ref{cond}), we have
\begin{eqnarray*}
\sup_{t\in\S}\Big| \hat{f}_n(t\mid x) - f(t\mid x)\Big| = O_{a.s.}(h_K^\beta + h_H^\nu) + O_{a.s.}\left(\left( \frac{\log n}{n\phi(h_K)}\right)^{1/2} \right).
\end{eqnarray*}
\end{proposition}
\begin{proof}{\it of Proposition \ref{density}}.

\noindent Following a similar decompositions and steps as in the proof  of Propositions \ref{convFhat}, we can easily prove the result of Proposition \ref{density}.
\end{proof}

\begin{proof}{\it of Theorem \ref{theoremQ} }

\noindent Using a Taylor expansion of the function $F(\widehat{q}_{n,\a}(x) \mid x)$ around $q_{\a}(x)$ we get:
\begin{eqnarray}\label{Taylor1}
F(\widehat{q}_{n,\a}(x) \mid x) - F(q_{\a}(x) \mid x) = \left( \widehat{q}_{n,\a}(x)-  q_\a(x) \right) f\left(q^\star_{n,\a}(x)\mid x\right),
\end{eqnarray}
where $q^\star_{n,\a}(x)$ lies between $q_\a(x)$ and $\widehat{q}_{n,\a}(x)$. Equation (\ref{Taylor1}) shows that from the asymptotic behavior of  $F(\widehat{q}_{n,\a}(x) \mid x) - F(q_{\a}(x) \mid x)$ as $n$ goes to infinity, it is easy to obtain asymptotic results for the sequence $\left( \widehat{q}_{n,\a}(x) -  q_\a(x)\right)$.

Subsequently, considering the statement (\ref{ineq}) together with the statement (\ref{Taylor1}), we obtain
\begin{eqnarray}\label{r1}
\Big| \widehat{q}_{n,\a}(x)-  q_\a(x) \Big| \times \Big|f\left(q^\star_{n,\a}(x)\mid x\right)\Big| = O\left(\sup_{t\in\S} \Big|F(t \mid x) - \widehat{F}_n(t \mid x)\Big| \right)
\end{eqnarray}
Using Lemma \ref{convNR}, condition (A3) and the statement (\ref{r1}), we get
\begin{eqnarray}\label{r2}
\Big| \widehat{q}_{n,\a}(x)-  q_\a(x) \Big| = O\left(\sup_{t\in\S} \Big|F(t \mid x) - \widehat{F}_n(t \mid x)\Big| \right),
\end{eqnarray}
which is enough, while considering Proposition \ref{convFhat}, to complete the proof of Theorem \ref{theoremQ}.

\end{proof}

\begin{proof}{\it of Theorem \ref{normF}}

\noindent To proof our result we need to introduce the following decomposition
\begin{eqnarray*}
\widehat{F}_n(t \mid x) - F(t \mid x) := J_{1,n} + J_{2,n} + J_{3,n},
\end{eqnarray*}
where $J_{1,n} := \widehat{F}_n(t \mid x) - \widetilde{F}_n(t \mid x)$, $J_{2,n} := \widetilde{F}_n(t \mid x) - \overline{\widetilde{F}}_n(t \mid x)$ and $J_{3,n} := \overline{\widetilde{F}}_n(t \mid x) - F(t\mid x).$
First, we establish that $J_{1,n}$ and $J_{3,n}$ are negligible, as $n\rightarrow\infty$, whereas $J_{2,n}$ is asymptotically normal.
Observe that the term $J_{1,n} :=  \widehat{F}_n(t \mid x) - \widetilde{F}_n(t \mid x)$ has been studied in Lemma \ref{equiv}, then we have 
\begin{eqnarray}\label{J1}
J_{1,n} = O_{a.s.}\left(\sqrt{\frac{\log_2n}{n}}\right).
\end{eqnarray}
On the other hand the term $J_{3,n} :=  \overline{\widetilde{F}}_n(t \mid x) - F(t\mid x)$ is equal to $B_n(x,t)$ which uniformly converges almost surely to zero (with rate $h_K^\beta + h_H^\nu$) by the Lemma \ref{biais} given in the Appendix. Then, we have
\begin{eqnarray}\label{J3}
J_{3,n} = O_{a.s.}(h_K^\beta + h_H^\nu).
\end{eqnarray}

Now, let us consider the term $J_{2,n}$ which will provide us the asymptotic normality. For this end, we consider the following decomposition of the term $J_{2,n}$.
\begin{eqnarray}\label{J2decomp}
J_{2,n} &=& \widetilde{F}_n(t \mid x) - \overline{\widetilde{F}}_n(t \mid x) \nonumber\\
&:= & \frac{Q_n(x,t) + R_n(x,t)}{\ell_n(x)},
\end{eqnarray}
where $Q_n(x,t) := [\widetilde{F}_n(x,t)- \overline{\widetilde{F}}_n(x,t)] - F(t\mid x) (\ell_n(x) - \overline{\ell}_n(x))$ and $R_n(x,t) := -B_n(x,t) (\ell_n(x) - \overline{\ell}_n(x))$, where $B_n(x,t):= \frac{\overline{\widetilde{F}}_n(x,t)}{\overline{\ell}_n(x)} - F(t\mid x)$.
Using results of Lemma \ref{biais}, we have, for any fixed $x\in E$, $B_n(x, t)$ and therefore $R_n(x, t)$ converge almost surely to zero when $n$ goes to infinity. Thus, the asymptotic normality will be provided by the term $Q_n(x, t)$ which is treated by the Lemma \ref{normalityQ} below.
\begin{lemma}\label{normalityQ}
Suppose that assumptions (A1)-(A3), (A5), (A8)-(A9) hold, and condition (\ref{cond}) satisfied, then we have
\begin{eqnarray*}
\sqrt{n\phi(h_K)}\;\; Q_n(x,t) \stackrel{\mathcal{D}}{\longrightarrow} \mathcal{N}\left( 0, \sigma^2(x,t)\right), \quad \mbox{as}\; n\rightarrow\infty,
\end{eqnarray*}
where $\sigma^2(x,t)$ is defined in Theorem \ref{normF}.
\end{lemma}

\noindent Finally, the proof of Theorem \ref{normF} can be achieved by considering equations (\ref{J1}), (\ref{J3}) and Lemma \ref{normalityQ}.

\end{proof}

\begin{proof} of Theorem \ref{normality}

\noindent Using the Taylor expansion of $\widehat{F}_n(\cdot\mid x)$ around $q_\a(x)$ we get:
\begin{eqnarray}\label{Taylor}
\widehat{F}_n(q_{\a}(x) \mid x) - F(q_{\a}(x) \mid x) = \left( q_\a(x) - \widehat{q}_{n,\a}(x)\right) \hat{f}_n\left(q^\star_{n,\a}(x)\mid x\right),
\end{eqnarray}
where $q^\star_{n,\a}(x)$ lies between $q_\a(x)$ and $\widehat{q}_{n,\a}(x)$. 

Then, by combining the consistency result given by Lemma \ref{convNR} and Proposition \ref{density}, we get
\begin{eqnarray}\label{Taylor2}
\widehat{q}_{n,\a}(x) - q_\a(x) &=&  -\frac{\widehat{F}_n(q_{\a}(x) \mid x) - F(q_{\a}(x) \mid x)}{f\left(q_{\a}(x)\mid x\right)}
% &=& - \frac{\left(\widehat{F}_n(q_{\a}(x) \mid x) - \widetilde{F}_n(q_{\a}(x) \mid x)\right) + \left(\widetilde{F}_n(q_{\a}(x) \mid x) - \overline{\widetilde{F}}_n(q_{\a}(x) \mid x) \right) + \left( \overline{\widetilde{F}}_n(q_{\a}(x) \mid x) - F(q_\a(x)\mid x) \right)}{F^{(1)}\left(q_{\a}(x)\mid x\right)} + o_{a.s.}(1)\nonumber\\
\end{eqnarray}
Finally, the combination of equation (\ref{Taylor2}) and Theorem \ref{normF} allows us to finish the proof of Theorem \ref{normality}.
\end{proof}

\begin{proof}{\it of Corollary \ref{corr}}

First, observe that 
\begin{eqnarray*}
& &\frac{M_{1,n} \;\hat{f}_n(\widehat{q}_{\a,n}(x)\mid x)}{\sqrt{M_{2,n}}} \sqrt{\frac{n F_{x,n}(h_K)}{\a \left( \bar{G}^{-1}_n(\widehat{q}_{\a,n}(x)) - \a\right)}} \; \left(\widehat{q}_{n,\a}( x) - q_{\a}(x)\right)  = \\
& &\frac{M_{1,n}\sqrt{M_2}}{M_1\sqrt{M_{2,n}}} \sqrt{\frac{nF_{x,n}(h_K) (\bar{G}^{-1}(q_\a(x))-\a)}{(\bar{G}_n^{-1}(\widehat{q}_{n,\a}(x))-\a)f_1(x) n\phi(h_K)}}\frac{\hat{f}_n(\widehat{q}_{n,\a}(x)\mid x)}{f(q_\a(x)\mid x)}\times \\
& & \frac{M_1}{\sqrt{M_2}}\sqrt{\frac{n\phi(h)f_1(x)}{\a (\bar{G}^{-1}(q_\a(x))-\a)}} f(q_\a(x)\mid x) \left( \widehat{q}_{n,\a}(x) - q_\a(x)\right)
\end{eqnarray*}

We have form Theorem \ref{normality} 
$$
 \frac{M_1}{\sqrt{M_2}}\sqrt{\frac{n\phi(h)f_1(x)}{\a (\bar{G}^{-1}(q_\a(x))-\a)}} f(q_\a(x)\mid x) \left( \widehat{q}_{n,\a}(x) - q_\a(x)\right) \stackrel{\mathcal{D}}{\longrightarrow} \mathcal{N}(0,1).
$$
Using results given by \cite{laib2010}, we have $M_{1,n} \stackrel{\mathbb{P}}{\longrightarrow} M_1$, $M_{2,n} \stackrel{\mathbb{P}}{\longrightarrow} M_2$ and $\frac{F_{x,n}(h_K)}{\phi(h_K)f_1(x)} \stackrel{\mathbb{P}}{\longrightarrow} 1$ as $n \rightarrow \infty.$

If in addition, we consider Proposition \ref{density}, Lemma \ref{convNR}, the consistency of $\bar{G}_n^{-1}(\cdot)$ to $\bar{G}^{-1}(\cdot)$ (given in \cite{deheuvels} ), one gets
$$
\frac{M_{1,n}\sqrt{M_2}}{M_1\sqrt{M_{2,n}}} \sqrt{\frac{nF_{x,n}(h_K) (\bar{G}^{-1}(q_\a(x))-\a)}{(\bar{G}_n^{-1}(\widehat{q}_{n,\a}(x))-\a)f_1(x) n\phi(h_K)}}\frac{\hat{f}_n(\widehat{q}_{n,\a}(x)\mid x)}{f(q_\a(x)\mid x)} \stackrel{\mathbb{P}}{\longrightarrow} 1, \quad \mbox{as}\; n\rightarrow\infty. 
$$ 
Therefore, the proof of Corollary \ref{corr} is achieved.
\end{proof}

\section*{Appendix }
\subsection*{Intermediate results for strong consistency}
\begin{proof}{\it of Lemma \ref{diff}}

Define the {``\it pseudo-conditional bias"} of the conditional distribution function estimate of $Y_i$ given $X=x$ as 
$$
B_n(x,t) = \frac{\overline{\widetilde{F}}_n(x,t)}{\overline{\ell}_n(x)} - F(t\mid x).
$$
Consider now the following quantites
$$
R_n(x,t) = -B_n(x,t) (\ell_n(x) - \overline{\ell}_n(x)),
$$
and
$$
Q_n(x,t) = (\widetilde{F}_n(x,t)- \overline{\widetilde{F}}_n(x,t)) - F(t\mid x) (\ell_n(x) - \overline{\ell}_n(x)).
$$
It is then clear that the following decomposition holds
\begin{eqnarray}\label{decomF}
\widetilde{F}_n(t \mid x) - F(t \mid x) = B_n(x,t) + \frac{R_n(x,t) + Q_n(x,t)}{\ell_n(x)}.
\end{eqnarray}

\begin{remark}\label{remarkQ}
Using statement (\ref{Q}) and Lemma \ref{lem2}, one can easily get, for all $x\in E$,
$$
\sup_{t\in \S}|Q_n(x,t)| = O_{a.s.}\left( \sqrt{\frac{\log n}{n\phi(h_K)}}\right).
$$ 
\end{remark}
Finally, the combination of results given in Lemma \ref{biais} and Remark \ref{remarkQ} achieves the proof of Lemma \ref{diff}.

\end{proof}

\begin{lemma}\label{biais}
Assume that hypothesis (A1)-(A5), (A7) and the condition (\ref{cond})
\begin{eqnarray}\label{biaiscond}
\sup_{t\in\S}|B_n(x,t)| = O_{a.s.}\left(h_K^\beta + h_H^\nu\right)  
\end{eqnarray}
\begin{eqnarray}\label{R}
\sup_{t\in\S}|R_n(x,t)| = O_{a.s.}\left((h_K^\beta+h_H^\nu)\left(\frac{\log n}{n \phi(h_K)}\right)^{1/2}\right).
\end{eqnarray}
\end{lemma}
\begin{proof}{\it of Lemma \ref{biais}}

Recall that

$$
B_n(x,t) = \frac{\overline{\widetilde{F}}_n(x,t)}{\overline{\ell}_n(x)} - F(t\mid x) = \frac{\overline{\widetilde{F}}_n(x,t) - \overline{\ell}_n(x) F(t\mid x)}{\overline{\ell}_n(x)}.
$$

\noindent By double conditioning with respect to the $\sigma$-field $\G_{i-1}$ and $T_i$ and using assumption (A4) and the fact that $\1_{\{ T_i \leq C_i\}} \varphi(Y_i) = \1_{\{ T_i \leq C_i\}} \varphi(T_i)$, we get
\begin{eqnarray*}
\overline{\widetilde{F}}_n(x,t) &=& \frac{1}{n\E(\Delta_1(x))} \sum_{i=1}^n \E\left\{ \Delta_i(x) \E\left[ \delta_i \bar{G}^{-1}(Y_i)\; H(h_{H}^{-1}(t-Y_i)) \mid \G_{i-1}, T_i \right] \mid \F_{i-1}\right\}\\
&=&  \frac{1}{n\E(\Delta_1(x))} \sum_{i=1}^n \E\left\{  \Delta_i(x) \E\left[ \delta_i \bar{G}^{-1}(Y_i)\; H(h_{H}^{-1}(t-Y_i))\mid X_i, T_i \right] \mid \F_{i-1}\right\}\\
&=& \frac{1}{n\E(\Delta_1(x))} \sum_{i=1}^n  \E\left\{ \bar{G}^{-1}(T_i)\; H(h_{H}^{-1}(t-T_i)) \Delta_i(x) \E\left[ \1_{\{ T_i \leq C_i\}} \mid X_i, T_i\right]\mid \F_{i-1}\right\}\\
&=&  \frac{1}{n\E(\Delta_1(x))} \sum_{i=1}^n  \E\left\{  \Delta_i(x)\; H(h_{H}^{-1}(t-T_i)) \mid \F_{i-1}\right\}\\
%&=& \frac{1}{n\E(\Delta_1(x))} \sum_{i=1}^n \E\left\{\Delta_i(x) \mbox{F}(t\mid X_i)\mid \F_{i-1} \right\}
\end{eqnarray*}
Then, by a double conditioning with respect to $\G_{i-1}$, we have
\begin{eqnarray*}
\overline{\widetilde{F}}_n(x,t) - \overline{\ell}_n(x) F(t\mid x) &=& \frac{1}{n\E(\Delta_1(x))} \sum_{i=1}^n \E\left\{ \Delta_i(x) [\E(H(h_{H}^{-1}(t-T_i)) \mid X_i) - F(t\mid x)] \mid \F_{i-1}\right\}\\
%&\leq& C_x h^\beta \overline{\ell}_n(x).
\end{eqnarray*}
Now, because of conditions (A3) and (A5), we get 
\begin{eqnarray}\label{star}
\Big|\E(H(h_{H}^{-1}(t-T_i)) \mid X_i) - F(t\mid x)\Big| \leq C_x \int_\mathbb{R} H^{(1)}(u) \left( h_K^{\beta} + |u|^\nu h_H^\nu\right) du. 
\end{eqnarray}
Therefore, we obtain 
\begin{eqnarray*}
\overline{\widetilde{F}}_n(x,t) - \overline{\ell}_n(x) F(t\mid x) &=& O_{a.s.}\left(h_K^\beta + h_H^\nu\right) \times \frac{1}{n\E(\Delta_1(x))} \sum_{i=1}^n \E\left\{\Delta_i(x) \mid \F_{i-1} \right\}\\
&=& O_{a.s.}\left(h_K^\beta + h_H^\nu\right) \times \overline{\ell}_n(x).
\end{eqnarray*}
Similarly as in Lemma \ref{lem2}, it is easily seen that $\overline{\ell}_n(x) = O_{a.s.}(1)$. Thus, we obtain $\overline{\widetilde{F}}_n(x,t) - \overline{\ell}_n(x) F(t\mid x) = O_{a.s.}\left(h_K^\beta + h_H^\nu\right).$

\noindent The second part of Lemma \ref{biais} follows easily from the fact that $R_n(x,t) = -B_n(x,t) (\ell_n(x) - \overline{\ell}_n(x))$, the statement of (\ref{biaiscond}) and Lemma \ref{lem2}, we get
$$
\sup_{t\in\S}|R_n(x,t)| = O_{a.s.}\left( (h_K^\beta+h_H^\nu) \left( \frac{\log n}{n\phi(h_K)}\right)^{1/2}\right).
$$
\end{proof}

\begin{lemma}\label{Qbis}
Assume that (A1)-(A2) and (A4)-(A7) are satisfied. Then, for any $x \in E$, we have
\begin{eqnarray}\label{Q}
\sup_{t\in\S}|\widetilde{F}_n(x,t)- \overline{\widetilde{F}}_n(x,t)| = O_{a.s.}\left( \left( \frac{\log n}{n\phi(h_K)}\right)^{1/2}\right).
\end{eqnarray}
\end{lemma}

\begin{proof}{\it of Lemma \ref{Qbis}}

Observe that
\begin{eqnarray*}
\widetilde{F}_n(x,t)- \overline{\widetilde{F}}_n(x,t) = \frac{1}{n\E(\Delta_1(x))} \sum_{i=1}^n L_{i,n}(x,t),
\end{eqnarray*}

where $L_{i,n}(x,t) = \delta_i \bar{G}^{-1}(Y_i) H(h_{H}^{-1}(t-Y_i)) \Delta_i(x) - \E\left(\delta_i \bar{G}^{-1}(Y_i) H(h_{H}^{-1}(t-Y_i)) \Delta_i(x)\mid \F_{i-1} \right)$ is a martingale difference. Therefore, we can use Lemma \ref{fact1} to obtain an exponential upper bound relative to the quantity $\widetilde{F}_n(x,t)- \overline{\widetilde{F}}_n(x,t).$ Let us now check the conditions under which one can obtain the mentioned exponential upper bound. In this respect, for any $p\in \mathbb{N}-\{ 0\}$, observe that
$$
L_{n,i}^p(x,t) = \sum_{k=0}^p C_p^k \left(\frac{\delta_i}{\bar{G}(Y_i)}\; H(h_{H}^{-1}(t-Y_i)) \Delta_i(x) \right)^k (-1)^{p-k} \left[\E\left(\frac{\delta_i}{\bar{G}(Y_i)}\; H(h_{H}^{-1}(t-Y_i)) \Delta_i(x)\mid \F_{i-1} \right) \right]^{p-k}
$$
In view of condition (A4), $\left[\E\left(\delta_i \bar{G}^{-1}(Y_i)\; H(h_{H}^{-1}(t-Y_i)) \Delta_i(x)\mid \F_{i-1} \right) \right]^{p-k}$ is $\F_{i-1}$-measurable, it follows then that
\begin{eqnarray*}
\E(L_{i,n}^p(x,t) \mid \F_{i-1}) = \sum_{k=0}^p C_p^k \E\left[\left(\delta_i \bar{G}^{-1}(Y_i)\; H(h_{H}^{-1}(t-Y_i)) \Delta_i(x) \right)^k \mid \F_{i-1}\right] (-1)^{p-k}\times\\ \left[\E\left(\delta_i \bar{G}^{-1}(Y_i)\; H(h_{H}^{-1}(t-Y_i)) \Delta_i(x)\mid \F_{i-1} \right) \right]^{p-k}.
\end{eqnarray*}
Thus,
\begin{eqnarray*}
\Big|\E(L_{i,n}^p(x,t) \mid \F_{i-1}) \Big| &\leq& \sum_{k=0}^p C_p^k \E\left[\left|\delta_i \bar{G}^{-1}(Y_i)\; H(h_{H}^{-1}(t-Y_i)) \Delta_i(x) \right|^k \mid \F_{i-1}\right] \times\\ 
& &\left[\E\left(\left|\delta_i \bar{G}^{-1}(Y_i)\; H(h_{H}^{-1}(t-Y_i)) \Delta_i(x) \right| \mid \F_{i-1} \right) \right]^{p-k}.
\end{eqnarray*}
Making use of Jensen inequality, one can write

\begin{eqnarray*}
\E\left[\left|\delta_i \bar{G}^{-1}(Y_i)\; H(h_{H}^{-1}(t-Y_i)) \Delta_i(x) \right|^k \mid \F_{i-1}\right]  \left[\E\left(\left|\delta_i \bar{G}^{-1}(Y_i)\; H(h_{H}^{-1}(t-Y_i)) \Delta_i(x) \right| \mid \F_{i-1} \right) \right]^{p-k} &\leq&\\
 \E\left(\left|\delta_i \bar{G}^{-1}(Y_i)\; H(h_{H}^{-1}(t-Y_i)) \Delta_i(x) \right|^k \mid \F_{i-1}\right)   \E\left(\left|\delta_i \bar{G}^{-1}(Y_i)\; H(h_{H}^{-1}(t-Y_i)) \Delta_i(x) \right|^{p-k} \mid \F_{i-1} \right)  & &
\end{eqnarray*}
Observe now that for any $m\geq 2$
\begin{eqnarray*}
\E\left(\left|\delta_i \bar{G}^{-1}(Y_i)\; H(h_{H}^{-1}(t-Y_i)) \Delta_i(x) \right|^m \mid \F_{i-1} \right) &\leq& \frac{1}{(\bar{G}(\tau))^{m-1}}\E\left( H^m(h_{H}^{-1}(t-Y_i)) \Delta_i^m(x)  \mid \F_{i-1} \right) \\
&\leq&  \frac{1}{(\bar{G}(\tau))^{m-1}} \E\left( \Delta_i^m(x) g_{m}(X_i,t) \mid \F_{i-1} \right).
\end{eqnarray*}
In view of assumption (A6), we have
\begin{eqnarray*}
\E\left(\left|\frac{\delta_i}{\bar{G}(Y_i)}  H(h_{H}^{-1}(t-Y_i)) \Delta_i(x) \right|^m \mid \F_{i-1} \right) & \leq & \frac{1}{(\bar{G}(\tau))^{m-1}}\left\{ \E\left( \Delta_i^m(x) |g_{m}(X_i,t)-g_{m}(x,t)| \mid \F_{i-1} \right) +\right.\\
& &\left. g_{m}(x,t) \E[\Delta_i^m(x) \mid \F_{i-1}]\right\}\\
&\leq &  \frac{1}{(\bar{G}(\tau))^{m-1}} \E[\Delta_i^m(x) \mid \F_{i-1}] \left[ \sup_{x^\prime\in B(x,h)} |g_{m}(x^\prime,t)-g_{m}(x,t)| \right.\\
& & \left.+ g_{m}(x,t)\right]\\
&\leq &  \frac{C_0}{(\bar{G}(\tau))^{m-1}} \E[\Delta_i^m(x) \mid \F_{i-1}],
\end{eqnarray*}
where $C_0$ is a positive constant.  By using Lemma \ref{lem1}, conditions (A2)(ii) and (A2)(iii),  whenever the kernel $K$ and the function $\tau_0$ are bounded by constants $a_1$ and $c_1$ respectively, we get, for $m=k$,
$$
\E[\Delta_i^k(x) \mid \F_{i-1}] = \phi(h_K) M_k f_{i,1}(x) + O_{a.s.}(g_{i,x}(h_K)) = c_1 \phi(h_K) a_1^k f_{i,1}(x) + O_{a.s.}(g_{i,x}(h_K)).
$$
Similarly, with $m=p-k$, we get
$$
\E\left(\left|\delta_i \bar{G}^{-1}(Y_i)\; H(h_{H}^{-1}(t-Y_i)) \Delta_i(x) \right|^{p-k} \mid \F_{i-1} \right) = c_1 \phi(h_K) a_1^{p-k} f_{i,1}(x) + O_{a.s.}(g_{i,x}(h_K)).
$$
Therefore,
\begin{eqnarray*}
\E\left(\left|\delta_i \bar{G}^{-1}(Y_i)\; H(h_{H}^{-1}(t-Y_i)) \Delta_i(x) \right|^k \mid \F_{i-1}\right)   \E\left(\left|\delta_i \bar{G}^{-1}(Y_i)\; H(h_{H}^{-1}(t-Y_i)) \Delta_i(x) \right|^{p-k} \mid \F_{i-1} \right) =\\
 c_1^2 a_1^p \phi(h_K)^2 f_{i,1}^2(x) + O_{a.s.}(g_{i,x}(h_K))\phi(h_K) f_{i,1}(x)(a_1^k + a_1^{p-k}) + O_{a.s.}(g^2_{i,x}(h_K)).
\end{eqnarray*}
Since $f_{i,1}(x)$ is almost surely bounded by a deterministic quantity $b_i(x)$, $g_{i,x}(h_K) \leq \phi(h_K)$ almost surely and $\phi(h_K)^2 < \phi(h_K)$, for $n$ sufficiently large, then following the same arguments as in the proof of Lemma 5 in \cite{laib2011}, one may write almost surely,
$$
\Big|\E(L_{i,n}^p(x,t) \mid \F_{i-1}) \Big| = p! C^{p-2}[C_2 \phi(h_K) f_{i,1}(x) + O_{a.s.}(g_{i,x}(h_K))] \leq p! C^{p-2} \phi(h_K) [M b_i(x) +1],
$$
where $C=2\max(1, a_1^2)$ and $C_2$ a positive constant. By taking $d_i^2 = \phi(h_K) [M b_i(x) +1]$, then $D_n=\sum_{i=1}^n d_i^2$ and by assumptions (A2)(ii) and (A2)(v) one gets $n^{-1}D_n = \phi(h_K) [M D(x) + o_{a.s.}(1)]$ as $n\rightarrow\infty,$ we now use the Lemma \ref{fact1} with $D_n=O_{a.s.}(n\phi(h_K))$ and $S_n=\sum_{i=1}^n L_{n,i}(x,t).$ Thus, for any $\epsilon_0 >0$, we can easily get
\begin{eqnarray*}
\P\left( \Big| \widetilde{F}_n(x,t)- \overline{\widetilde{F}}_n(x,t) \Big|> \epsilon_0 \sqrt{\frac{\log n}{n\phi(h_K)}} \right) &=& \P\left( \Big| \sum_{i=1}^n L_{n,i}(x,t)\Big| > n\E(\Delta_1(x))\epsilon_0 \sqrt{\frac{\log n}{n\phi(h_K)}}\right)\\
&\leq&2\exp\left\{ -\frac{(n\E(\Delta_1(x))\epsilon_0)^2\frac{\log n}{n\phi(h_K)}}{2D_n + 2Cn\E(\Delta_1(x))\epsilon_0\sqrt{\frac{\log n}{n\phi(h_K)}}}\right\}\\
&\leq& 2\exp\left\{ -C_1\epsilon_0^2\log n\right\} = \frac{2}{n^{C_1\epsilon_0^2}},
\end{eqnarray*}
where $C_1$ is a positive constant. Therefore, choosing $\epsilon_0$ large enough, we obtain
$$
\sum_{n\geq 1}\P\left( \Big| \widetilde{F}_n(x,t)- \overline{\widetilde{F}}_n(x,t) \Big|> \epsilon_0 \sqrt{\frac{\log n}{n\phi(h_K)}}\right) <\infty.
$$
Finally, we achieve the proof by Borel-Cantelli Lemma.
\end{proof}

\begin{proof}{\it of Lemma \ref{equiv} }

\noindent From (\ref{pseudo}) and (\ref{estF2}) we have

\begin{eqnarray*}
 \Bigl\lvert \widehat{F}_n(t \mid x) - \widetilde{F}_n(t \mid x)\Bigl\lvert &\leq&\frac{1}{n\mathbb{E} \left[\Delta_1 (x)\right] \ell_n(x)}
\sum_{i=1}^n\left|\delta_i\Delta_i (x)\; \; H(h_{H}^{-1}(t-Y_i))\left(\frac{1}{\bar{G}(Y_i)}-\frac{1}{\bar{G}_n (Y_i)}\right)\right|\\
&\leq& \frac{\sup_{t\in S}|\bar{G_n}(t)-\bar{G}(t)|}{\bar{G}_n(\tau)}\; \widetilde{F}_n(t\mid x)\\
%&=& \frac{\sup_{t\in S}|\bar{G_n}(t)-\bar{G}(t)|}{\bar{G}(\tau_F)}\; \left[B_n(x,t) - F(t\mid x)\right]
\end{eqnarray*}

\noindent Since $\bar{G}(\tau)> 0,$ in conjuction with the Srong Law of Large Numbers (SLLN) and the Law of the Iterated Logarithm (LIL) on the censoring law (see Theorem 3.2 of \cite{cai}, the result is an immediate consequence of Lemmas \ref{diff}.

\end{proof}

\subsection*{Intermediate results for asymptotic normality }

\begin{proof} of Lemma \ref{normalityQ}

\noindent Let us denote by 
\begin{eqnarray}\label{eta}
\eta_{ni} = \left( \frac{\phi(h_K)}{n}\right)^{1/2} \left( \frac{\delta_i}{\bar{G}(Y_i)} H(h_H^{-1}(t - Y_i)) - F(t\mid x)\right)\frac{\Delta_i(x)}{\E(\Delta_1(x))},
\end{eqnarray}
and define $\xi_{ni} := \eta_{ni} - \E[\eta_{ni} \mid \F_{i-1}].$ It is easy seen that
 \begin{eqnarray}\label{diffmart}
 (n\phi(h_K))^{1/2} Q_n(x,t) = \sum_{i=1}^n \xi_{ni},
 \end{eqnarray}

where, for any fixed $x\in E$, the summands in (\ref{diffmart}) from a triangular array of stationary martingal differences with respect to the $\sigma$-field $\F_{i-1}$. This allows us to apply the Central Limit Theorem for discrete-time arrays of real-valued martingales (see, \cite{hall}, page 23) to establish the asymptotic normality of $Q_n(x,t)$. Therefore, we have to establish the following statements:
\begin{itemize}
\item[(a)] $\sum_{i=1}^n \E[\xi_{ni}^2 \mid \F_{i-1}] \stackrel{\mathcal{\mathbb{P}}}{\longrightarrow} \sigma^2(x,t),$
\item[(b)] $n \E[\xi_{ni}^2 \; \1_{[|\xi_{ni}| > \epsilon]}] = o(1)$ holds for any $\epsilon >0$ (Lindeberg condition).
\end{itemize}
\subsubsection*{Proof of part (a).} Observe that
\begin{eqnarray*}
\Big| \sum_{i=1}^n \E[\eta_{ni}^2 \mid \F_{i-1}] - \sum_{i=1}^n \E[\xi_{ni}^2 \mid \F_{i-1}]\Big| \leq \sum_{i=1}^n (\E[\eta_{ni} \mid \F_{i-1}])^2.
\end{eqnarray*}
Using Lemma \ref{lem1} and inequality (\ref{star}), we obtain
\begin{eqnarray}
\Big| \E[\eta_{ni} \mid \F_{i-1}] \Big| &=& \frac{1}{\E(\Delta_1(x))} \left( \frac{\phi(h_K)}{n}\right)^{1/2} \Big| \E\left[ \Delta_i(x)\left( \frac{\delta_i}{\bar{G}(Y_i)} H(h_H^{-1}(t - Y_i))- F(t\mid x)\right)\mid \F_{i-1}\right] \Big|\nonumber\\
&=& O_{a.s.}\left(h_K^\beta + h_H^\nu\right) \left( \frac{\phi(h_K)}{n}\right)^{1/2} \left( \frac{f_{i,1}(x)}{f_1(x)} + O_{a.s.}\left( \frac{g_{i,x}(h_K)}{\phi(h_K)}\right)\right) 
\end{eqnarray}
Then, by (A2)(ii)-(iii), we get
\begin{eqnarray*}
\sum_{i=1}^n (\E[\eta_{ni} \mid \F_{i-1}])^2 = O_{a.s.}\left(\left(h_K^\beta + h_H^\nu\right)^2 \phi(h_K)\right).
\end{eqnarray*}
The statement of (a) follows then if we show that
\begin{eqnarray}\label{proba}
\lim_{n\rightarrow \infty}\sum_{i=1}^n \E[\eta_{ni}^2\mid \F_{i-1}] \stackrel{\mathcal{\mathbb{P}}}{\longrightarrow} \sigma^2(x, t).
\end{eqnarray}
To prove (\ref{proba}), observe that, using assumption (A8), we have
\begin{eqnarray*}
\sum_{i=1}^n \E[\eta_{ni}^2\mid \F_{i-1}] &=& \frac{\phi(h_K)}{n(\E(\Delta_1(x)))^2} \sum_{i=1}^n \E\left\{\Delta_i^2(x) \left( \frac{\delta_i}{\bar{G}(Y_i)} H(h_H^{-1}(t - Y_i)) - F(t \mid x)\right)^2\mid \F_{i-1}\right\} \\
&=&  \frac{\phi(h_K)}{n(\E(\Delta_1(x)))^2} \sum_{i=1}^n \E\left\{\Delta_i^2(x) \E\left[ \left(\frac{\delta_i}{\bar{G}(Y_i)} H(h_H^{-1}(t- Y_i)) - F(t \mid X_i) \right)^2 \mid X_i\right] \mid \F_{i-1}\right\}.
%= K_{n1} + K_{n2},
\end{eqnarray*}

\noindent Using the definition of the conditional variance, we have
\begin{eqnarray}
\E\left[ \left(\frac{\delta_i}{\bar{G}(Y_i)} H(h_H^{-1}(t- Y_i)) - F(t \mid X_i) \right)^2 \mid X_i\right] &=& \mbox{var}\left[\frac{\delta_i}{\bar{G}(Y_i)} H(h_H^{-1}(t- Y_i))  \mid X_i\right] + \nonumber\\
& & \left[ \E\left( \frac{\delta_i}{\bar{G}(Y_i)} H(h_H^{-1}(t- Y_i)) \mid X_i\right) - F(t\mid x)\right]^2 \nonumber\\
&=:& K_{n1} + K_{n2}
\end{eqnarray}

By the use of a double conditioning with respect to $T_i$, inequality (\ref{star}), assumption (A3) and Lemma \ref{lem1}, we can easily get
\begin{eqnarray}\label{termK2}
 \frac{\phi(h_K)}{n(\E(\Delta_1(x)))^2} \sum_{i=1}^n \E\left\{\Delta_i^2(x) K_{n2}\mid \F_{i-1}\right\} = O_{a.s.}\left((h_K^\beta + h_H^\nu)^2 \right) \left[ \frac{M_2}{M_1^2}\frac{1}{f_1(x)} + o_{a.s.}(1)\right].
\end{eqnarray}
Let us now examine the term $K_{n1}$,
\begin{eqnarray*}
K_{n1} &=& \E\left[ \frac{\delta_i}{\bar{G}^2(Y_i)} H^2\left(\frac{t-Y_i}{h_H}\right) \mid X_i\right] - \left[ \E\left( \frac{\delta_i}{\bar{G}(Y_i)} H\left(\frac{t-Y_i}{h_H}\right)\mid X_i\right)\right]^2\\
&=& \i_1 + \i_2.
\end{eqnarray*}

\noindent The first term of the last equality can be developed as follow, 
\begin{eqnarray*}
\i_1 &=& \E\left[ H^2\left(\frac{t-Y_i}{h_H}\right) \frac{1}{\bar{G}(Y_i)} \mid X_i\right]\\
&=& \int_\mathbb{R} H^2\left(\frac{t-z}{h_H}\right) \frac{1}{\bar{G}(z)} f(z\mid X_i) dz\\
&=&  \int_\mathbb{R} H^2(v) \frac{1}{\bar{G}(t-h_Hv)} dF(t-h_Hv\mid X_i). 
\end{eqnarray*}
By the first order Taylor expansion of the function $\bar{G}^{-1}(\cdot)$ around zero one gets
\begin{eqnarray*}
\i_1 &:=& \int_\mathbb{R} H^2(v) \frac{1}{\bar{G}(t)} dF(t-h_Hv\mid X_i)+ \frac{h_H}{\bar{G}^2(t)}   \int_\mathbb{R} v H(v) \bar{G}^{(1)}(t^\star) dF(t-hv\mid X_i) + o(1)\\
&=:& \i^\prime_1 +  \i^\prime_2 ,
\end{eqnarray*}
where $t^\star$ is between $t$ and $t-h_Hv.$

\noindent Under assumption (A9), we have $\i^\prime_2 \leq h_H^2 \frac{\sup_{u\in \mathbb{R}}|G^{(1)}(u)|}{\bar{G}^2(t)}\int_\mathbb{R} v f(t-h_H v \mid X_i) dv$. Then, using assumption (A3), we get $\i^\prime_2  = O(h_H^2).$

On the other hand, by integrating by part we have
\begin{eqnarray*}
\i^\prime_1 &=& \frac{1}{\bar{G}(t)} \int_{\mathbb{R}} 2 H^\prime(v) H(v) F(t-h_Hv\mid X_i) dv\\
&=& \frac{1}{\bar{G}(t)} \int_{\mathbb{R}} 2 H^\prime(v) H(v) \left(F(t-h_Hv\mid X_i) - F(t\mid x)\right) dv +  \frac{1}{\bar{G}(t)} \int_{\mathbb{R}} 2 H^\prime(v) H(v)  F(t\mid x) dv. 
\end{eqnarray*}
Then, under assumption (A3), we get $\i^\prime_1 = \frac{F(t\mid x)}{\bar{G}(t)} + O(h_K^\beta + h_H^\nu)$ and therefore 
\begin{eqnarray}
\i_1 =  \frac{F(t\mid x)}{\bar{G}(t)} + O(h_K^\beta + h_H^\nu) + O(h_H^2).
\end{eqnarray}
Finally, we get
\begin{eqnarray*}
\frac{\phi(h_K)}{n(\E(\Delta_1(x)))^2} \sum_{i=1}^n \E \left \{\Delta_i^2(x) K_{n1} \mid \F_{i-1}\right\} &=& \left( O(h_K^\beta + h_H^\nu) + O(h_H^2) + \frac{F(t \mid x)}{\bar{G}(t)} - (F(t \mid x))^2 \right) \times\\
& & \frac{\phi(h_K)}{n(\E(\Delta_1(x)))^2} \sum_{i=1}^n \E\left(\Delta_i^2(x) \mid F_{i-1} \right).
\end{eqnarray*}
Then, $\lim_{n\rightarrow\infty} \frac{\phi(h_K)}{n(\E(\Delta_1(x)))^2} \sum_{i=1}^n \E \left \{\Delta_i^2(x) K_{n2} \mid \F_{i-1}\right\} =0$, almost surely and
\begin{eqnarray*}
\lim_{n\rightarrow\infty} \frac{\phi(h_K)}{n(\E(\Delta_1(x)))^2} \sum_{i=1}^n \E \left \{\Delta_i^2(x) K_{n1} \mid \F_{i-1}\right\} = \left(\frac{F(t \mid x)}{\bar{G}(t)} - (F(t \mid x))^2 \right) \times \left( \frac{M_2}{M_1^2}\frac{1}{f_1(x)}\right).
\end{eqnarray*}
Therefore,
\begin{eqnarray*}
\sum_{i=1}^n \E[\eta_{ni}^2\mid \F_{i-1}] =  \left(\frac{F(t \mid x)}{\bar{G}(t)} - (F(t \mid x))^2 \right) \times \left( \frac{M_2}{M_1^2}\frac{1}{f_1(x)}\right)=:\sigma(x,t).
\end{eqnarray*}

%%%%%%%%%%%
%%%%%%%%%%%%

This is complete the Proof of part (a).
\subsubsection*{Proof of part (b).} The Lindeberg condition results from Corollary 9.5.2 in \cite{chow} which implies that $n\E[\xi_{ni}^2 \1_{[|\xi_{ni} > \epsilon|]}] \leq 4n\E[\eta_{ni}^2 \1_{[|\eta_{ni}| > \epsilon/2]}].$ Let $a>1$ and $b>1$ such that $1/a + 1/b =1$. Making use of H\"older and Markov inequalities one can write, for all $\epsilon > 0$,
\begin{eqnarray*}
\E[\eta_{ni}^2 \1_{[|\eta_{ni}| > \epsilon/2]}] \leq \frac{\E|\eta_{ni}|^{2a}}{(\epsilon/2)^{2a/b}}.
\end{eqnarray*}
Taking $C_0$ a positive constant and $2a =2+\delta$ (with $\delta$ as in (A8)), using the condition (A8) and a double conditioning, we obtain
\begin{eqnarray*}
4n\E[\eta_{ni}^2 \1_{[|\eta_{ni}| > \epsilon/2]}] &\leq& C_0 \left( \frac{\phi(h_K)}{n}\right)^{(2+\delta)/2} \frac{n}{(\E(\Delta_1(x)))^{2+\delta}} \times\\
& & \E\left(\left[\Big| \frac{\delta_i}{\bar{G}(Y_i)} H(h_H^{-1}(t - Y_i)) - F(t\mid x) \Big| \Delta_i(x) \right]^{2+\delta}\right)\\
&\leq & C_0 \left( \frac{\phi(h_K)}{n}\right)^{(2+\delta)/2} \frac{n}{(\E(\Delta_1(x)))^{2+\delta}} \E\left[ (\Delta_i(x))^{2+\delta} \overline{W}_{2+\delta}(t\mid X_i)\right]\\
&\leq & C_0  \left( \frac{\phi(h_K)}{n}\right)^{(2+\delta)/2} \frac{n\E[(\Delta_1(x))^{2+\delta}]}{(\E(\Delta_1(x)))^{2+\delta}}\;\; (|\overline{W}_{2+\delta}(t\mid x)| + o(1))
\end{eqnarray*}
Now, using Lemma \ref{lem1}, we get
\begin{eqnarray*}
4n\E[\eta_{ni}^2 \1_{[|\eta_{ni}| > \epsilon/2]}]  \leq C_0 (n\phi(h_K))^{-\delta/2} \frac{(M_{2+\delta})f_1(x) + o(1)}{(M_1^{2+\delta}f_1^{2+\delta}(x)) + o(1)} \; (|\overline{W}_{2+\delta}(t\mid x)| + o(1)) = O((n\phi(h_K))^{-\delta/2}).
\end{eqnarray*}
This completes the proof of part (b) and therefore the proof of Lemma \ref{normalityQ}.
\end{proof}
%%%%%%%%%%%%%%%
%% convergence ss vitesse
%%%%%%%%%%%%%%%
%\clearpage
%
%

\end{document}